\documentclass[smallextended]{svjour3}       % onecolumn (second format)
\smartqed  % flush right qed marks, e.g. at end of proof
\usepackage{graphicx}
\usepackage{amsmath}
\usepackage{amssymb}
\usepackage{enumerate}
\usepackage[colorlinks,citecolor=blue,linkcolor=blue]{hyperref}
\usepackage[numbers,sort&compress]{natbib}
\usepackage[justification=centering]{caption}
\newcommand{\C}{{\mathbb{C}}}
\newcommand{\ket}[1]{|{#1}\rangle}

\begin{document}

\title{Constructions of  unextendible entangled bases%\thanks{Grants or other notes
%about the article that should go on the front page should be
%placed here. General acknowledgments should be placed at the end of the article.}
}
%\subtitle{Do you have a subtitle?\\ If so, write it here}

\titlerunning{Constructions of  unextendible entangled bases}        % if too long for running head

\author{Fei Shi        \and
       Xiande Zhang   \and
        Yu Guo
}

\authorrunning{F. Shi et al.} % if too long for running head

\institute{F. Shi  \at
              School of Cyber Security,
              University of Science and Technology of China, Hefei, 230026, People's Republic of China. \\
              \email{shifei@mail.ustc.edu.cn}\\
              X. Zhang \at
              School of Mathematical Sciences,
              University of Science and Technology of China, Hefei, 230026, People's Republic of China. \\
              \email{drzhangx@ustc.edu.cn}\\
              Y. Guo \at
              Institute of Quantum Information Science, Shanxi Datong University, Datong 037009,  People's Republic of China.\\
              %Tel.: +123-45-678910\\
              %Fax: +123-45-678910\\
              \email{guoyu3@aliyun.com} }

\date{Received: date / Accepted: date}
% The correct dates will be entered by the editor

\maketitle

\begin{abstract}
We provide several constructions of special unextendible entangled bases with fixed Schmidt number $k$ (SUEB$k$) in $\C^{d}\otimes \C^{d'}$ for $2\leq k\leq d\leq d'$.  We generalize the space decomposition method in Guo [Phys. Rev. A 94, 052302 (2016)], by proposing a systematic way of constructing new SUEB$k$s in $\C^{d}\otimes \C^{d'}$ for $2\leq k < d \leq d'$ or $2\leq k=d< d'$. In addition, we give a construction of a $(pqdd'-p(dd'-N))$-number SUEB$pk$ in $\C^{pd}\otimes \C^{qd'}$ from an $N$-number SUEB$k$ in $\C^{d}\otimes \C^{d'}$ for $p\leq q$  by using permutation matrices.  We also connect a $(d(d'-1)+m)$-number UMEB in $\C^{d}\otimes \C^{d'}$ with an unextendible partial Hadamard matrix $H_{m\times d}$ with $m<d$, which extends the result in [Quantum Inf. Process. \textbf{16}(3), 84 (2017)].

% \PACS{PACS code1 \and PACS code2 \and more}
% \subclass{MSC code1 \and MSC code2 \and more}

\end{abstract}

\section{Introduction}

\label{sec:intro}

In 1999, Bennett~\textit{et al}. found that unextendible product basis (UPB), a set of incomplete orthonormal product states
whose complementary space has no product states, also can displays nonlocality~\cite{CHDP}.
  Consequently,
the notion of unextendible basis has been explored extensively.
It is shown that UPB also can be used to construct bound entangled states~\cite{DTPJ,Aop,Spo,Skl,Sbb}.
In 2009, Bravyi and Smolin introduced the concept of unextendible maximally entangled bases (UMEB),
which is a set of incomplete orthonormal maximally entangled states whose complementary
space has no maximally entangled states. UMEBs can be used to construct examples of states for which
one-copy entanglement of assistance (EoA) is strictly smaller than the asymptotic, and can be
used to find quantum channels that are unital but not convex mixtures of unitary operations~\cite{Brav,Chen}.
Recently, Guo et al.~\cite{Guo14,GuoM} proposed the concept of special unextendible entangled bases
with fixed Schmidt number $k$ (SUEB$k$), which extends the definitions of both UPB and UMEB.
An SUEB$k$ is a set of incomplete orthonormal  special entangled states with Schmidt number $k$
in $\C^{d}\otimes \C^{d'}$ ($d\leq d'$), i.e., with the same Schmidt coefficients,
whose complementary subspace has no special entangled states with Schmidt number $k$. An SUEB$k$ is a UPB when $k=1$ and is a UMEB when $k=d$.

One of the main topics in this field is the existence of different unextendible bases~\cite{Brav,Wang17,LiMS,Zhang19,Wang14,Guo16,Guo14}.
 Bravyi and Smolin gave a $6$-number UMEB in $\C^{3}\otimes \C^{3}$ and a 12-number
UMEB in $\C^{4}\otimes \C^{4}$ \cite{Brav}. Wang et al. showed that there are
UMEBs in $\C^{d}\otimes \C^{d}$ except for $d=p$ or $2p$, where $p$ is a prime
and $p\equiv 3\pmod 4$ \cite{Wang17}. For $d<d'$, Li et al. showed that UMEBs exist in
any $\C^{d}\otimes \C^{d'}$ and gave explicit constructions \cite{LiMS}.
There are several recursive constructions for UMEBs, such as a construction of a
$(pqd^{2}-p(d^{2}-N))$-number UMEB in $\C^{pd}\otimes\C^{qd}$ ($p\leq q$)
from an $N$-number UMEB in $\C^{d}\otimes \C^{d}$, and a construction of a
$(pqd^{2}-d(pq-N))$-number UMEB in $\C^{pd}\otimes \C^{qd}$ ($p\leq q$)
from an $N$-number UMEB in $\C^{p}\otimes \C^{q}$ \cite{Zhang19,Wang14,Guo16}.
It was shown that SUEB$k$s exist in any bipartite system when $2\leq k<d\leq d'$ \cite{Guo14}.

In this paper, we develop more constructions of SUEB$k$s in $\C^{d}\otimes \C^{d'}$ for  $2\leq k\leq d\leq d'$. In Section \ref{sec:EBk}, we provide a systematic way of constructing SUEB$k$s from a special entangled basis with Schmidt number $k$, and show that there are more new SUEB$k$s in $\C^{d}\otimes \C^{d'}$ for $2\leq k < d \leq d'$ or $2\leq k=d<d'$. In Section \ref{sec:SUEBk}, we give a recursive construction of a $(pqdd'-p(dd'-N))$-number SUEB$pk$ in $\C^{pd}\otimes \C^{qd'}$ from an $N$-number SUEB$k$ in $\C^{d}\otimes \C^{d'}$ for $p\leq q$  by using permutation matrices. Especially, it generates a $(pqdd'-p(dd'-N))$-number UMEB in $\C^{pd}\otimes \C^{qd'}$ from an $N$-number UMEB in $\C^{d}\otimes \C^{d'}$ for $p\leq q$. In Section \ref{sec:had}, we show that there is a $(d(d'-1)+m)$-number UMEB in $\C^{d}\otimes \C^{d'}$ if there exists an unextendible partial Hadamard matrix $H_{m\times d}$. Our constructions generalize and improve most results in \cite{Chen,LiMS,Wang14,Guo14,Guo16,Wang17,ZhanJ,Zhang19}, and provide new SUEB$k$s and UMEBs.

\section{Definition and preliminary}
\label{sec:def}

Assume that $1\leq k \leq d \leq d'$ in this paper. Let $[n]$ denote the set $\{1,2,\cdots,n\}$, $[n]^{*}$ denote the set $\{0,1,\cdots,n-1\}$, and $\{a,a,\ldots,a\}_k$ denote a multiset of $k$ numbers of $a$. The Schmidt decomposition of  a pure state  $\ket{\psi}\in\C^{d}\otimes\C^{d'}$ \cite{Niel}: $\ket{\psi}=\sum_{i=0}^{k-1}\lambda_{i}\ket{i}\ket{i'}$, where $\{\ket{i}\}$ and $\{\ket{i'}\}$ are orthonormal sets of $\C^{d}$ and $\C^{d'}$, respectively. Then the Schmidt number of $\ket{\psi}$, denoted by $Sr(|\psi\rangle)$, is $k$. If all the Schmidt coefficients are  $\{\frac{1}{\sqrt{k}},\frac{1}{\sqrt{k}},\frac{1}{\sqrt{k}}\cdots,\frac{1}{\sqrt{k}}\}_{k}$, then $\ket{\psi}$ is called a special entangled state with Schmidt number $k$. A special entangled state with Schmidt number $k$ is a product state  when $k=1$; and it is a maximally entangled state  when $k=d$.

Let $\mathcal{M}_{d\times d'}$ be the space of all $d\times d'$ complex matrices equipped with an inner product defined by $(A,B)=\text{Tr}(A^{\dag}B)$. There is a one-to-one relation between the space $\C^{d}\otimes\C^{d'}$ and the space $\mathcal{M}_{d\times d'}$ \cite{Guo15,GuoM}:
\begin{equation}
\begin{aligned}
|\psi_{i}\rangle=&\sum_{k,l}a_{k,l}^{(i)}|k\rangle|l'\rangle\in \C^{d}\otimes \C^{d'} \Longleftrightarrow A_{i}=(a_{k,l}^{(i)})\in \mathcal{M}_{d\times d'},\\
&Sr(|\psi_{i}\rangle)=\text{rank}(A_{i}), \ \langle\psi_{i}|\psi_{j}\rangle=\text{Tr}(A_{i}^{\dag}A_{j}),
\end{aligned}
\end{equation}
where $\{\ket{k}\}$ and $\{\ket{l'}\}$ are the standard computational bases of $\C^{d}$ and $\C^{d'}$, respectively. A $d\times d'$ matrix is called a {\it $k$-singular-value-1 matrix} if its nonzero singular values are $\{1,1,1\cdots,1\}_{k}$. Then $\ket{\psi_{i}}$ is a special entangled state with Schmidt number $k$ if and only if $\sqrt{k}A_{i}$ is a $k$-singular-value-1 matrix. Specially, when $d=d'$, $|\psi_{i}\rangle$ is a maximally entangled state if and only if $\sqrt{d}A_{i}$ is a unitary matrix \cite{Guo16}.

\begin{definition} \cite{Guo15} A set of special entangled states with Schmidt number $k$ $\{|\psi_{i}\rangle\}_{i=1}^{dd'}$ of $\C^{d}\otimes \C^{d'}$ is called a  special entangled basis with Schmidt number $k$ (SEB$k$) if  $\langle\psi_{i}|\psi_{j}\rangle=\delta_{ij}$.
\end{definition}

\begin{definition} \cite{GuoM,Guo14} A set of special entangled states with Schmidt number $k$  $\{|\psi_{i}\rangle: i\in{[n]},n<dd'\}$ of $\C^{d}\otimes \C^{d'}$ is called an $n$-number special unextendible  entangled basis with  Schmidt number $k$ (SUEB$k$) if
	\begin{enumerate}[(i)]
		\item $\langle\psi_{i}|\psi_{j}\rangle=\delta_{ij}$;
		\item If $\langle\psi_{i}|\phi\rangle=0$ for all $i\in[n]$, then $\ket{\phi}$ can not be a  special entangled state with Schmidt number $k$.
	\end{enumerate}
\end{definition}

Note that the condition (ii) in Definition 2.2 is a bit weaker than that in \cite{GuoM,Guo14}, where it states that ``if $\langle\psi_{i}|\phi\rangle=0$ for all $i\in[n]$, then $Sr(|\psi\rangle)\neq k$". Specially, an SEB$k$ is a product basis and  an SUEB$k$ is a UPB when $k=1$;  an SEB$k$ is a maximally entangled basis (MEB) and an SUEB$k$ is a UMEB  when $k=d$  \cite{Brav,Chen,CHDP,DTPJ}. Analogous to SEB$k$s and SUEB$k$s, we give the definitions regarding to $k$-singular-value-1 matrices.

\begin{definition} A set of $k$-singular-value-1 matrices $\{A_{i}\}_{i=1}^{dd'}$ of $\mathcal{M}_{d\times d'}$ is called a $k$-singular-value-1 Hilbert-Schmidt basis (SV1B$k$) if $\text{Tr}(A_{i}^{\dag}A_{j})=k\delta_{ij}$.
\end{definition}

\begin{definition} A set of $k$-singular-value-1 matrices $\{A_{i}: i\in[n],n<dd'\}$ of $\mathcal{M}_{d\times d'}$ is called an $n$-number unextendible $k$-singular-value-1 Hilbert-Schmidt basis (USV1B$k$) if
	\begin{enumerate}[(i)]
		\item $\text{Tr}(A_{i}^{\dag}A_{j})=k\delta_{ij}$;
		\item If $\text{Tr}(A_{i}^{\dag}B)=0$ for all $i\in[n]$, then $B$ can not be a $k$-singular-value-1 matrix.
	\end{enumerate}
\end{definition}

Due to the one-to-one relation, $\{|\psi_{i}\rangle\}$ is an SEB$k$ if and only if $\{A_{i}\}$ is an SV1B$k$; and $\{|\psi_{i}\rangle\}$ is an SUEB$k$ if and only if $\{A_{i}\}$ is a USV1B$k$.

\begin{lemma}\label{ebk} \cite{Guo15} If $k|dd'$, then there is an SEB$k$ in $\C^{d}\otimes \C^{d'}$, and consequently there is an SV1B$k$ in  $\mathcal{M}_{d\times d'}$.
\end{lemma}

Let $\mathcal{L}$ denote a subspace of $\mathcal{M}_{d\times d'}$, $\mathcal{L}^{\bot}$ denote the complementary space of $\mathcal{L}$, and $\oplus$ denote the direct sum. Inspired by Lemmas 2 and  3 of \cite{Guo16}, we get the following two lemmas.
\begin{lemma}\label{usv}  Let $\mathcal{M}_{d\times d'}=\mathcal{L}\oplus\mathcal{L}^{\bot}$. If there is an SV1B$k$ $\{A_{i}\}$ in $\mathcal{L}$ and there is no $k$-singular-value-1 matrix in $\mathcal{L}^{\bot}$, then $\{A_{i}\}$ is a USV1B$k$ in $\mathcal{M}_{d\times d'}$.
\end{lemma}
\begin{lemma}\label{usv1}  Let $\mathcal{M}_{d\times d'}=\mathcal{L}\oplus\mathcal{L}^{\bot}$. If there is an SV1B$k$
	$\{A_{i}\}$ in $\mathcal{L}$ and a USV1B$k$ $\{B_{i}\}$ in $\mathcal{L}^{\bot}$, then $\{A_{i}\}\cup\{B_{i}\}$ is a USV1B$k$ in $\mathcal{M}_{d\times d'}$.
\end{lemma}

\section{Constructions of SUEB$k$s from SEB$k$s}
\label{sec:EBk}

In this section, we introduce several constructions of SUEB$k$s from SEB$k$s in $\C^{d}\otimes \C^{d'}$.

In the following notations, each $\mathcal{L}$ defines a subspace of $\mathcal{M}_{d\times d'}$ which consists of all matrices having zero entries in the specified places. The subscripts of $\mathcal{L}^{(i)}$ denote the size of the submatrices consisting of stars for $i\in[5]$, and the size of the bottom right submatrix for $i=6$. Let
\begin{align*}
\mathcal{L}_{d\times (d'-i)}^{(1)}=\left(
\begin{matrix}
*      & \cdots &  *       & 0      &\cdots & 0      \\
\vdots & \ddots & \vdots   & \vdots &\ddots & \vdots \\
*      & \cdots &  *       & 0      &\cdots & 0
\end{matrix}
\right)_{d\times d'},\ \ \
\mathcal{L}_{d\times i}^{(2)}=\left(
\begin{matrix}
0      & \cdots &  0       & *      &\cdots & *      \\
\vdots & \ddots & \vdots   & \vdots &\ddots & \vdots \\
0     & \cdots &   0       & *      &\cdots & *
\end{matrix}
\right)_{d\times d'},%\tag{2}
\end{align*}
\begin{align*}
\mathcal{L}_{(d-i)\times d'}^{(3)}=\left(
\begin{matrix}
*      & \cdots &  * \\
\vdots & \ddots &\vdots \\
*      & \cdots &  * \\
0      & \cdots &  0 \\
\vdots & \ddots &\vdots \\
0      & \cdots &  0
\end{matrix}
\right)_{d\times d'},\ \ \
\mathcal{L}_{i\times d'}^{(4)}=\left(
\begin{matrix}
0      & \cdots &  0 \\
\vdots & \ddots &\vdots \\
0      & \cdots &  0 \\
*      & \cdots &  * \\
\vdots & \ddots &\vdots \\
*      & \cdots &  *
\end{matrix}
\right)_{d\times d'},% \tag{3}
\end{align*}
\begin{align*}
\mathcal{L}_{(d-i)\times (d'-t)}^{(5)}=\left(
\begin{matrix}
*      & \cdots &  *    &0 &\cdots &0              \\
\vdots & \ddots &\vdots &\vdots & \ddots &\vdots  \\
*      & \cdots &  *    &0 &\cdots &0               \\
0      & \cdots &  0    &0 &\cdots &0                \\
\vdots & \ddots &\vdots &\vdots & \ddots &\vdots     \\
0      & \cdots &  0    &0 &\cdots &0
\end{matrix}
\right)_{d\times d'} \text{  and   }
\mathcal{L}_{i\times t}^{(6)}=\left(
\begin{matrix}
0      & \cdots &  0    &* &\cdots &*              \\
\vdots & \ddots &\vdots &\vdots & \ddots &\vdots  \\
0      & \cdots &  0    &* &\cdots &*               \\
*      & \cdots &  *    &* &\cdots &*                \\
\vdots & \ddots &\vdots &\vdots & \ddots &\vdots     \\
*      & \cdots &  *    &* &\cdots &*
\end{matrix}
\right)_{d\times d'}, %\tag{4}
\end{align*}
then $\mathcal{M}_{d\times d'}$ has direct-sum decompositions $\mathcal{L}_{d\times (d'-i)}^{(1)}\oplus \mathcal{L}_{d\times i}^{(2)}$, $\mathcal{L}_{(d-i)\times d'}^{(3)}\oplus \mathcal{L}_{i\times d'}^{(4)}$ and $\mathcal{L}_{(d-i)\times (d'-t)}^{(5)}\oplus \mathcal{L}_{i\times t}^{(6)}$.

We apply Lemmas~\ref{ebk} and~\ref{usv} to the following four cases.
\begin{enumerate}
	\item[(1)]\label{case1} $k|d$. For any  $i\in[k-1]$ satisfying $d'-i\geq k$, use decomposition
	$$\mathcal{M}_{d\times d'}
	=\mathcal{L}_{d\times (d'-i)}^{(1)}\oplus \mathcal{L}_{d\times i}^{(2)}.$$%\eqno(5)
	%obviously, $\mathcal{L}_{d\times i}^{2}=(\mathcal{L}_{d\times (d'-i)}^{1})^{\bot}$.
	Since $k|d$, there is an SV1B$k$ in $\mathcal{L}_{d\times (d'-i)}^{(1)}$ from Lemma~\ref{ebk}. As the rank of any matrix in $\mathcal{L}_{d\times i}^{(2)}$ is no more than $i<k$, then there is no $k$-singular-value-1 matrix in $\mathcal{L}_{d\times i}^{(2)}$. Thus there is a $d(d'-i)$-number USV1B$k$ in $\mathcal{M}_{d\times d'}$ by Lemma~\ref{usv}.
\end{enumerate}

By similar arguments, we have the following three cases.
\begin{enumerate}
	\item[(2)] \label{case2} $k\nmid d$. Write $d=sk+r$ such that $0<r<k$. For any $t\in [k-r]^*$ satisfying $d'-t\geq k$, use
	$$\mathcal{M}_{d\times d'}=\mathcal{L}_{sk\times (d'-t)}^{(5)}\oplus \mathcal{L}_{r\times t}^{(6)}.$$%\eqno(6)
	Then there is an $sk(d'-t)$-number USV1B$k$ in $\mathcal{M}_{d\times d'}$.
	
	\item[(3)] \label{case3} $k|d'$. For any  $i\in[k-1]$ satisfying $d-i\geq k$, use decomposition
	$$\mathcal{M}_{d\times d'}
	=\mathcal{L}_{(d-i)\times d'}^{(3)}\oplus \mathcal{L}_{i\times d'}^{(4)}.$$%\eqno(7)
	
	Then there is a $d'(d-i)$-number USV1B$k$ in $\mathcal{M}_{d\times d'}$.
	
	\item[(4)] \label{case4} $k\nmid d'$. Write $d'=sk+r$ such that $0<r<k$. For any $t\in [k-r]^*$ satisfying $d-t\geq k$, use
	$$\mathcal{M}_{d\times d'}
	=\mathcal{L}_{(d-t)\times sk}^{(5)}\oplus \mathcal{L}_{t\times r}^{(6)}.$$%\eqno(8)
	Then there is an $sk(d-t)$-number USV1B$k$ in $\mathcal{M}_{d\times d'}$.
\end{enumerate}

For any $k,d$ and $d'$ satisfying the conditions of any of the above four cases, we can construct SUEB$k$s in $\C^{d}\otimes \C^{d'}$  from SEB$k$s. However, if $k=d=d'$ or $k=1$, none of the four cases are satisfied. So we can not use this method to construct UMEBs in $\C^{d}\otimes \C^{d}$ or UPBs in $\C^{d}\otimes \C^{d'}$.  We summarize them in the following theorem.
\begin{theorem}
	When $2\leq k <d \leq d'$ or $2\leq k=d<d'$,  there are four different classes of SUEBks as the cases (1)-(4) above.
\end{theorem}
\begin{example}
	There is a $28$-number SUEB$4$ in $\C^{6}\otimes \C^{7}$.
	
	Since $4\nmid 6$, we can get a $28$-number SUEB$4$ in $\C^{6}\otimes \C^{7}$  from  Case (2) with $t=0$.
	$$|\psi_{1}\rangle=\frac{1}{2}(|00\rangle+|11\rangle+|22\rangle+|33\rangle), \
	|\psi_{2}\rangle=\frac{1}{2}(|01\rangle+|12\rangle+|23\rangle+|34\rangle),$$
	$$|\psi_{3}\rangle=\frac{1}{2}(|02\rangle+|13\rangle+|24\rangle+|35\rangle), \ |\psi_{4}\rangle=\frac{1}{2}(|03\rangle+|14\rangle+|25\rangle+|36\rangle),$$
	$$|\psi_{5}\rangle=\frac{1}{2}(|04\rangle+|15\rangle+|26\rangle+|30\rangle), \ |\psi_{6}\rangle=\frac{1}{2}(|05\rangle+|16\rangle+|20\rangle+|31\rangle),$$
	$$|\psi_{7}\rangle=\frac{1}{2}(|06\rangle+|10\rangle+|21\rangle+|32\rangle), \
	|\psi_{8}\rangle=\frac{1}{2}(|00\rangle-|11\rangle+|22\rangle-|33\rangle),$$
	$$|\psi_{9}\rangle=\frac{1}{2}(|01\rangle-|12\rangle+|23\rangle-|34\rangle), \
	|\psi_{10}\rangle=\frac{1}{2}(|02\rangle-|13\rangle+|24\rangle-|35\rangle),$$  $$|\psi_{11}\rangle=\frac{1}{2}(|03\rangle-|14\rangle+|25\rangle-|36\rangle), \
	|\psi_{12}\rangle=\frac{1}{2}(|04\rangle-|15\rangle+|26\rangle-|30\rangle),$$  $$|\psi_{13}\rangle=\frac{1}{2}(|05\rangle-|16\rangle+|20\rangle-|31\rangle), \
	|\psi_{14}\rangle=\frac{1}{2}(|06\rangle-|10\rangle+|21\rangle-|32\rangle),$$
	$$|\psi_{15}\rangle=\frac{1}{2}(|00\rangle+|11\rangle-|22\rangle-|33\rangle), \
	|\psi_{16}\rangle=\frac{1}{2}(|01\rangle+|12\rangle-|23\rangle-|34\rangle),$$
	$$|\psi_{17}\rangle=\frac{1}{2}(|02\rangle+|13\rangle-|24\rangle-|35\rangle), \ |\psi_{18}\rangle=\frac{1}{2}(|03\rangle+|14\rangle-|25\rangle-|36\rangle),$$
	$$|\psi_{19}\rangle=\frac{1}{2}(|04\rangle+|15\rangle-|26\rangle-|30\rangle), \ |\psi_{20}\rangle=\frac{1}{2}(|05\rangle+|16\rangle-|20\rangle-|31\rangle),$$
	$$|\psi_{21}\rangle=\frac{1}{2}(|06\rangle+|10\rangle-|21\rangle-|32\rangle), \
	|\psi_{22}\rangle=\frac{1}{2}(|00\rangle-|11\rangle-|22\rangle+|33\rangle),$$
	$$|\psi_{23}\rangle=\frac{1}{2}(|01\rangle-|12\rangle-|23\rangle+|34\rangle), \
	|\psi_{24}\rangle=\frac{1}{2}(|02\rangle-|13\rangle-|24\rangle+|35\rangle),$$  $$|\psi_{25}\rangle=\frac{1}{2}(|03\rangle-|14\rangle-|25\rangle+|36\rangle),  \
	|\psi_{26}\rangle=\frac{1}{2}(|04\rangle-|15\rangle-|26\rangle+|30\rangle),$$  $$|\psi_{27}\rangle=\frac{1}{2}(|05\rangle-|16\rangle-|20\rangle+|31\rangle), \
	|\psi_{28}\rangle=\frac{1}{2}(|06\rangle-|10\rangle-|21\rangle+|32\rangle).$$
\end{example}
\begin{remark} Propositions 1 and 2 in \cite{Guo14} belong to our Case (4);  Propositions 3 and 4 in \cite{Guo14} belong to our Case (2);  Propositions 5 and 6 in \cite{Guo14} belong to our Case (1) and Case (3). But we give  more constructions than those in \cite{Guo14}, which can be easily seen from the constructions of SUEB$4$s in $\C^{7}\otimes \C^{12}$.  In fact, Case $(3)$ provides a $72$-number SUEB$4$, a $60$-number SUEB$4$ and a $48$-number SUEB$4$ in $\C^{7}\otimes \C^{12}$, while  \cite{Guo14} only gives a $48$-number SUEB$4$ in $\C^{7}\otimes \C^{12}$. Also, Our constructions  cover all of the results in \cite{Chen,LiMS,Guo16,ZhanJ}. See Table \ref{Our}.	
\end{remark}

\begin{table}[]
	\caption{Our results about SUEB$k$s in $\C^{d}\otimes \C^{d'}, 2\leq k\leq d\leq d'$.}
	\centering
	\renewcommand\tabcolsep{15.0pt}
	\begin{tabular}{cc}
		\hline\noalign{\smallskip}
		Condition & No. of SUEB$k$\\
		\noalign{\smallskip}\hline\noalign{\smallskip}
		$k|d$                           & $d(d'-i)$, $i\in[k-1], d'-i\geq k$ \\
		$d=sk+r, r\in[k-1]$             & $sk(d'-t)$, $t\in[k-r]^{*}, d'-t\geq k$\\
		$k|d'$                          & $d'(d-i)$, $i\in[k-1], d-i\geq k$ \\
		$d'=sk+r, r\in[k-1]$            & $sk(d-t)$,  $t\in[k-r]^{*}, d-t\geq k$ \\       	
		\noalign{\smallskip}\hline
	\end{tabular}\label{Our}
\end{table}

\section{SUEB$pk$s from SUEB$k$s}
\label{sec:SUEBk}

In this section, we give a general construction of SUEB$pk$s in $\C^{pd}\otimes \C^{qd'}$ from SUEB$k$s in $\C^{d}\otimes \C^{d'}$, where $p\leq q$. We introduce a combinatorial object first.
A permutation matrix is a square matrix that has exactly one entry of $1$ in each row and each column and $0$s elsewhere. By abusing of this concept, for nonsquare matrices, we call a $p\times q$ matrix ($p\leq q$) a permutation matrix if it has exactly one entry of 1 in each row and at most one entry of 1 in each column and 0s elsewhere. Let $J_{p\times q}$ be a $p\times q$ matrix with all entries being 1, then $J_{p\times q}$ can be decomposed as $J_{p\times q}=P_{0}+P_{1}+\cdots+P_{q-1}$, where each $P_{i}$ is a permutation matrix. For example, let $P_{l}=P_{0}T^{l}$, $l\in [q]^{*}$, where
\begin{align*}
P_{0}
&=\left(
\begin{matrix}
1      &0           &\cdots      & 0               &0        &\cdots    & 0       \\
0      &1 &\cdots      & 0               &0        &\cdots    & 0       \\
\vdots &\vdots      &\ddots      & \vdots          &\vdots&  &\vdots              \\
0      &0           &\ldots      &1 &0        &\cdots    & 0
\end{matrix}
\right)_{p\times q} \text{ and }\ \
T=\left(
\begin{matrix}
0      &1        &0      &\cdots       &0              \\
0      &0        &1      &\cdots       &0              \\
\vdots &\vdots   &\vdots &\ddots       &\vdots          \\
0      &0        &0      &\cdots       &1                \\
1      &0        &0      &\cdots       &0
\end{matrix}
\right)_{q\times q}.%,\tag{10}
\end{align*}

\begin{theorem}\label{thm1}
	If there is an $N$-number SUEB$k$ in $\C^{d}\otimes \C^{d'}$ constructed from Section \ref{sec:EBk}, and $k|dd'$, then there is a $(pqdd'-p(dd'-N))$-number SUEB$pk$ in $\C^{pd}\otimes \C^{qd'}$ for $p\leq q$.
\end{theorem}
\begin{proof}Given any decomposition  $J_{p\times q}=P_{0}+P_{1}+\cdots+P_{q-1}$. For any $l\in[q]^{*}$, $a\in [p]^{*}$, define a $p\times q$ matrix $Q_{l}^{a}$ by
	$$ Q_{l}^{a}(i,j)=\left\{
	\begin{array}{lll}
	0               &           & \text{if} \ \ P_{l}(i,j)=0,\\
	\xi_{p}^{a(i-1)}    &           & \text{if} \ \ P_{l}(i,j)=1,
	\end{array} \right. $$
	where $\xi_{p}=e^{\frac{2\pi i}{p}}$.
	For each matrix  $M\in \mathcal{M}_{pd\times qd'}$, write it as a block matrix $M=(M_{i,j})_{p\times q}$, where each $M_{i,j}$ is a $d\times d'$ submatrix. Then let $\mathcal{L}_l$ be a subspace of $\mathcal{M}_{pd\times qd'}$ which consists of all block matrices $M$ with  $M_{i,j}=0$ if $P_{l}(i,j)=0$. Then we have space decomposition,
	$$\mathcal{M}_{pd\times qd'}=\mathcal{L}_0\oplus\mathcal{L}_1\oplus\cdots\oplus\mathcal{L}_{q-1},$$
	such that $\dim\mathcal{L}_l=pdd'$ for all $l\in[q]^{*}$.
	
	Assume that $\{A_{j}\}_{j=1}^{N}$ is an $N$-number USV1B$k$ in $\mathcal{M}_{d\times d'}$ constructed from Section \ref{sec:EBk}. Since $k|dd'$, there is an SV1B$k$ in $\mathcal{M}_{d\times d'}$ by Lemma~\ref{ebk}. Denote it by $\{B_{s,t}\}$, where $s\in [d]$ and $t\in [d']$.  	For any $a\in[p]^{*}$, $l\in[q-1]$, $s\in [d]$,  $t\in [d']$ and $j\in [N]$, define
	$$C_{a,l}^{s,t}=Q_{l}^{a}\otimes B_{s,t} \text{ and } C_{a,0}^{j}=Q_{0}^{a}\otimes A_{j}.$$
	Obviously, $C_{a,l}^{s,t}\in \mathcal{L}_{l}$ for any $l\in[q-1]$ and $C_{a,0}^{j}\in \mathcal{L}_{0}$. Now we show that $\{C_{a,l}^{s,t}\}\cup\{C_{a,0}^{j}\}$ is a $(pqdd'-p(dd'-N))$-number USV1B$pk$ in $\mathcal{M}_{pd\times qd'}$.
	
	If the nonzero singular values of two matrices A and B are $\{1,1,\cdots 1\}_{p}$ and $\{1,1,\cdots 1\}_{k}$, respectively, then the singular values of $A\otimes B$ are $\{1,1,\cdots,1\}_{pk}$ \cite{scha}. Thus it is easy to see that $C_{k,l}^{s,t}$ and $C_{a,0}^{j}$ are $pk$-singular-value-1 matrices.
	It is also easy to see  that $\text{Tr}[(C_{\tilde{a},l}^{\tilde{s},\tilde{t}})^{\dag}C_{a,l}^{s,t}]=pk\delta_{\tilde{a}{a}}
	\delta_{\tilde{s}{s}}\delta_{\tilde{t}{t}}$ and $ \text{Tr}[(C_{\tilde{a},{0}}^{\tilde{j}})^{\dag}C_{a,0}^{j}]=pk\delta_{\tilde{a}{a}}
	\delta_{\tilde{j}{j}}$, where $a, \tilde{a}\in[p]^{*}$; $s, \tilde{s}\in [d]$;  $t, \tilde{t}\in [d']$; $j, \tilde{j}\in [N]$ and $l\in[q-1]$.  It follows that $\{C_{a,l}^{s,t}: s\in[d],t\in[d'],a\in[p]^*\}$ is an SV1B$pk$ in $\mathcal{L}_{l}$ for any  $l\in[q-1]$. We assert that $\{C_{a,0}^{j}\}$ is a USV1B$pk$ in $\mathcal{L}_{0}$. Given any $D=(D_{i,j})_{p\times q}\in \mathcal{L}_{0}$, let $D_{i}\triangleq D_{i,j}$ when $P_{0}(i,j)=1$ for $i\in[p]$. If $\text{Tr}(D^{\dag}C_{a,0}^{j})=0$ for all $a\in[p]^{*}$ and $j\in [N]$, then
	\begin{equation}
	\text{Tr}(D_{1}^{\dag}A_{j})+\xi_{p}^{a}\text{Tr}(D_{2}^{\dag}A_{j})+\cdots+\xi_{p}^{a(p-1)}
	\text{Tr}(D_{p}^{\dag}A_{j})=0.  \label{trt}
	\end{equation}
	This is equivalent to $WY_j=0$ for all $j\in [N]$, where
	\begin{align*}
	W
	&=\left(
	\begin{matrix}
	1      &1           &\cdots      &1                   \\
	1      &\xi_{p}     &\cdots      &\xi_{p}^{p-1}         \\
	\vdots &\vdots      &\ddots            &\vdots                \\
	1      &\xi_{p}^{p-1}&\cdots     &\xi_{p}^{(p-1)^{2}}
	\end{matrix}
	\right)_{p\times p}
	\text{ and }Y_j=\left(
	\begin{matrix}
	\text{Tr}(D_{1}^{\dag}A_{j})        \\
	\text{Tr}(D_{2}^{\dag}A_{j})       \\
	\vdots  \\
	\text{Tr}(D_{p}^{\dag}A_{j})
	\end{matrix}
	\right)_{p\times 1}.
	\end{align*}
	Since $\text{det}(W)\neq 0$, we know that $Y_j=0$ for all $j\in [N]$. It means that $\text{Tr}(D_{1}^{\dag}A_{j})=\text{Tr}(D_{2}^{\dag}A_{j})=\cdots=\text{Tr}(D_{p}^{\dag}A_{j})=0$
	for all $j\in [N]$. As $\{A_{j}\}_{j=1}^{N}$ is a USV1B$k$ in $\mathcal{M}_{d\times d'}$ constructed from Section \ref{sec:EBk}, we have $\text{rank}(D_{i})< k$ for all $i\in [p]$. Since the nonzero singular values of $D_{i}$, $i\in [p]$, form all the nonzero singular values of $D$, $D$ can not be a $pk$-singular-value-1 matrix. This shows that $\{C_{a,0}^{j}\}$ is a USV1B$pk$ in $\mathcal{L}_{0}$.
	
	We conclude that $\{C_{a,l}^{s,t}\}\cup\{C_{a,0}^{j}\}$ is a $(pqdd'-p(dd'-N))$-number USV1B$pk$ in $\C^{pd}\otimes \C^{qd'}$ by Lemma~\ref{usv1}.
\end{proof}

From the proof of Theorem~\ref{thm1}, we see that if $k=d$ and $D_{i}$, $i\in[p]$ satisfy Eq~(\ref{trt}) for all $j\in [N]$, then for any  USV1B$d$ $\{A_{j}\}_{j=1}^{N}$ (not necessarily from Section \ref{sec:EBk}), $D_{i}$ is not a $d$-singular-value-1 matrix for $i\in [p]$, and consequently $D$ is not a $pd$-singular-value-1 matrix.  Therefore, we have the following corollary.

\begin{corollary}\label{umeb} If there is an $N$-number UMEB in $\C^{d}\otimes \C^{d'}$, then there is a $(pqdd'-p(dd'-N))$-number UMEB in $\C^{pd}\otimes \C^{qd'}$ for $p\leq q$.
\end{corollary}

Our method is better than the method in \cite{Wang14,Zhang19,Guo16} since we can choose any decomposition of $J_{p\times q}$ into permutation matrices, and choose any different SV1B$d$s  $\{B_{s,t}\}$ in $\mathcal{M}_{d\times d'}$.
\begin{example}\label{u32}
	There is a $32$-number UMEB in $\C^{4}\otimes \C^{9}$ constructed from a $4$-number UMEB in $\C^{2}\otimes \C^{3}$.
	
	Let $\{A_i\}_{i=1}^4$ be a $4$-number USV1B$2$ in $\mathcal{M}_{2\times 3}$ that is constructed from Section \ref{sec:EBk} Case $(1)$ with $i=1$:
	\begin{align*}
	&A_{1}=\left(
	\begin{matrix}
	1     &0     &0      \\
	0    &1&0
	\end{matrix}\right),
	&A_{2}=\left(
	\begin{matrix}
	1      &0       &0    \\
	0      &-1&0
	\end{matrix}
	\right),\
	\end{align*}
	\begin{align*}
	&A_{3}=\left(
	\begin{matrix}
	0     &1      &0       \\
	1     &0&0
	\end{matrix}
	\right),
	&A_{4}=\left(
	\begin{matrix}
	0     &1      &0      \\
	-1     &0&0
	\end{matrix}\right).
	\end{align*}
	Let
	\begin{align*}
	&B_{s,t}=\left(
	\begin{matrix}
	1     &0          &0                \\
	0     &(-1)^{s}     &0              \\
	\end{matrix}\right)
	\left(
	\begin{matrix}
	0     &1          &0          \\
	0     &0          &1           \\
	1     &0          &0
	\end{matrix}\right)^{t},
	\end{align*}
	where $s\in[2]^{*}$, $t\in[3]^{*}$. It is easy to see that $\{B_{s,t}\}$ is an SV1B$2$ in $\mathcal{M}_{2\times 3}$.
	
	Let $J_{2\times 3}=P_{0}+P_{1}+P_{2}$,  $P_{l}=P_{0}T^{l}$, $l\in [3]^{*}$, where
	\begin{align*}
	&P_{0}=\left(
	\begin{matrix}
	1     &0          &0              \\
	0     &1     &0
	\end{matrix}
	\right), \ \ T=\left(
	\begin{matrix}
	0     &1         &0          \\
	0     &0         &1          \\
	1     &0         &0
	\end{matrix}\right).
	\end{align*}
	For any $l\in[3]^{*}$, $a\in [2]^{*}$, define a $2\times 3$ matrix $Q_{l}^{a}$ by
	$$ Q_{l}^{a}(i,j)=\left\{
	\begin{array}{lll}
	0               &           & \text{if} \ \ P_{l}(i,j)=0,\\
	(-1)^{a(i-1)}    &           & \text{if} \ \ P_{l}(i,j)=1.
	\end{array} \right.  $$
	Then for each $s\in[2]^{*}$, $t\in[3]^{*}$ and $j\in[4]$,
	
	\begin{align*}
	&C_{0,1}^{s,t}=\left(
	\begin{matrix}
	0                   &B_{s,t}       &0              \\
	0                   &0             &B_{s,t}
	\end{matrix}
	\right),
	&C_{1,1}^{s,t}=\left(
	\begin{matrix}
	0                   &B_{s,t}       &0              \\
	0                   &0             &-B_{s,t}
	\end{matrix}
	\right),
	\end{align*}
	\begin{align*}
	&C_{0,2}^{s,t}=\left(
	\begin{matrix}
	0                   &0             &B_{s,t}       \\
	B_{s,t}             &0             &0
	\end{matrix}
	\right),
	&C_{1,2}^{s,t}=\left(
	\begin{matrix}
	0                   &0             &B_{s,t}       \\
	-B_{s,t}            &0             &0
	\end{matrix}
	\right),
	\end{align*}
	
	\begin{align*}
	&C_{0,0}^{j}=\left(
	\begin{matrix}
	A_{j}         &0             &0        \\
	0             &A_{j}         &0
	\end{matrix}
	\right),
	&C_{1,0}^{j}=\left(
	\begin{matrix}
	A_{j}         &0       &0      \\
	0             &-A_{j}  &0
	\end{matrix}
	\right).
	\end{align*}
	
	Thus $\{C_{0,1}^{s,t}, C_{1,1}^{s,t}, C_{0,2}^{s,t}, C_{1,2}^{s,t}, C_{0,0}^{j}, C_{1,0}^{j}:  s\in[2]^{*}, t\in[3]^{*},j\in[4]\}$ is a $32$-number USV1B$4$ in $\mathcal{M}_{4\times 9}$.
\end{example}

\begin{remark} Theorem~\ref{thm1} gives a unified recursive construction for SUEB$k$s, including the special case about UMEBs when $k=d$ in Corollary~\ref{umeb}. In fact, Corollary~\ref{umeb} generalizes all the recursive constructions about UMEBs  in \cite{Guo16,Zhang19,Wang14} (see Table~\ref{Rec}):  Theorem 1 in \cite{Guo16,Wang14},
	is a special case when $d=d'$ and $p=q$; Theorem 1 in \cite{Zhang19} is the case when $d=d'$; and Theorem 2 in \cite{Zhang19} is the case when $p=q$.  Corollary~\ref{umeb}  can obtain more new examples than the constructions in  \cite{Zhang19}, see Example~\ref{u32}.
\end{remark}

\begin{remark}
	Corollary~\ref{umeb}  can also provide new examples of UMEB in $\C^{d}\otimes \C^{d'}$ that are different from the ones constructed in Section \ref{sec:EBk}. This  can be easily seen from Example~\ref{u32}. Note that we can also get a $32$-number UMEB in $\C^{4}\otimes \C^{9}$ either  from Case $(1)$ with $i=1$ or Case $(4)$ with $t=0$ in Section \ref{sec:EBk}. For both cases, the Schmidt numbers of the states are no more than $1$ in the  complementary subspace of these UMEBs. However, there is a state with Schmidt number $2$ in the complementary subspace of the UMEB in Example~\ref{u32}.
\end{remark}

\begin{table}[]
	\caption{Recursive constructions for UMEBs ($p\leq q, d\leq d'$)}
	\centering
	\renewcommand\tabcolsep{12.0pt}
	\begin{tabular}{ccc}
        \hline\noalign{\smallskip}
		Condition & No. of UMEB & Reference\\
        \noalign{\smallskip}\hline\noalign{\smallskip}
		$N$-.UMEB in $\C^{d}\otimes \C^{d}$   &$((qd)^{2}-q(d^{2}-N))$-.UMEB in $\C^{qd}\otimes \C^{qd}$            & \cite{Wang14,Guo16}\\
		$N$-.UMEB in $\C^{d}\otimes \C^{d}$ &$(pqd^{2}-p(d^{2}-N))$-.UMEB in $\C^{pd}\otimes \C^{qd}$                      & \cite{Zhang19}\\
		$N$-.UMEB in $\C^{p}\otimes \C^{q}$ &$(pqd^{2}-d(pq-N))$-.UMEB in  $\C^{pd}\otimes \C^{qd}$                      & \cite{Zhang19}\\
		$N$-.UMEB in $\C^{d}\otimes \C^{d'}$ &$(pqdd'-p(dd'-N))$-.UMEB in  $\C^{pd}\otimes \C^{qd'}$                      & This paper\\
        \noalign{\smallskip}\hline
	\end{tabular}\label{Rec}
\end{table}

\section{UMEBs from partial Hadamard matrices}
\label{sec:had}

In \cite{Wang17}, the authors gave a construction of UMEBs in $\C^{d}\otimes \C^{d}$ from partial Hadamard matrices. In this section, we generalize this construction to UMEBs in $\C^{d}\otimes \C^{d'}$ with $d\leq d'$.

A Hadamard matrix is a complex square matrix with entries in the unit circle $T$ whose rows are pairwise orthogonal. It is called a partial Hadamard matrix when the number of rows is less than the number of columns. As in Section~\ref{sec:SUEBk}, there is a matrix decomposition $J_{d\times d'}=P_{0}+P_{1}+\cdots P_{d'-1}$, $d\leq d'$. For any $l\in[d']^{*}$, $a\in [d]^{*}$, define a $d\times d'$ matrix $Q_{l}^{a}$ by
$$ Q_{l}^{a}(i,j)=\left\{
\begin{array}{lll}
0               &           & \text{if} \ \ P_{l}(i,j)=0\\
\xi_{d}^{a(i-1)}    &           & \text{if} \ \ P_{l}(i,j)=1
\end{array} \right. .$$
Let $\mathcal{L}_{l}$ be the subspace of $\mathcal{M}_{d\times d'}$ which consists of all matrices in $M_{d\times d'}$ with $m_{i,j}=0$ if $P_{l}(i,j)=0$. Then
$$\mathcal{M}_{d\times d'}=\mathcal{L}_0\oplus\mathcal{L}_1\oplus\cdots\oplus\mathcal{L}_{d'-1},$$
and $\dim\mathcal{L}_l=d$ for all $l\in[d']^{*}$.  Obviously, $Q_{l}^{a}\in\mathcal{L}_l$ for any $l\in[d']^{*}$.  It is easy to see that $\{Q_{l}^a: a\in[d]^*\}$ is an SV1B$d$ in $\mathcal{L}_{l}$ for any  $l\in[d'-1]$. Let $Z_{0}=\{Q_{l}^{a}: l\in[d'-1],a\in[d]^{*}\}$. Similar to the method in Theorem~\ref{thm1}, if there is a USV1B$d$ $Z_{1}$ in $\mathcal{L}_0$, then $Z_{0}\cup Z_{1}$ is a USV1B$d$ in $\mathcal{M}_{d\times d'}$ by Lemma~\ref{usv1}.

Now we construct a USV1B$d$ $Z_{1}$ in $\mathcal{L}_{0}$ from a partial Hadamard matrix.
Given an $m\times d$ partial Hadamard matrix $H_{m\times d}=(h_{i,j})$ with $m<d$, define a $d\times d'$ matrix $H_{y}$, $y\in [m]$, by
$$H_{y}(i,j)=\left\{
\begin{array}{lll}
0               &           & \text{if} \ \ P_{0}(i,j)=0,\\
h_{y,i}   &           & \text{if} \ \ P_{0}(i,j)=1,
\end{array} \right.  $$
and let
$$Z_{1}=\{H_{y}\}_{y=1}^{m},$$
then each $H_{y}\in\mathcal{L}_{0}$  and it is a $d$-singular-value-1 matrix.  Further, $$\text{Tr}(H_y^{\dag}H_{y'})=\sum_{l=1}^{d}\overline{h}_{y,l}h_{y',l}=d\delta_{yy'}$$ for all $y,y'\in [m]$ by the definition of Hadamard matrices. So we only need the unextendibility of $Z_1$, which is equivalent to the unextendibility of the Hadamard matrix $H$.

\begin{theorem}\label{pop}
	Given a partial Hadamard matrix $H_{m\times d}$, then $Z_{0}\cup Z_{1}$ is a $(d(d'-1)+m)$-number UMEB in $\mathcal{M}_{d\times d'}$ if and only if $H_{m\times d}$ can not be extended to an $(m+1)\times d$ partial Hadamard matrix.
\end{theorem}

\begin{example}\label{had1}
	There is a $(5(d'-1)+3)$-number UMEB in $\C^{5}\otimes \C^{d'}$ for all $d'\geq 5$.
	
	Let
	\begin{align*}
	H_{3\times 5}=(h_{i,j})
	&=\left(
	\begin{matrix}
	1      &1           &1      & 1               &1               \\
	1      &-1          &1      & \omega          &\omega^{2}       \\
	\frac{\sqrt{5}+i}{\sqrt{6}} &i      &\frac{-\sqrt{5}+i}{\sqrt{6}}     &\frac{\sqrt{6}\omega^{2} i+(\omega-1)i}{3}        &\frac{\sqrt{6}\omega i+(\omega^{2}-1)i}{3}
	\end{matrix}\right),\textsc{}
	\end{align*}
	where $\omega=e^{\frac{2\pi i}{3}}$, then it is a partial Hadamard matrix that can not be extended to a $4\times 5$ partial Hadamard matrix \cite{Wang17}. Let $J_{5\times d'}=P_{0}+P_{1}+\cdots +P_{d'-1}$, $P_{l}=P_{0}T^{l}$, $l\in [d']^{*}$, where
	\begin{align*}
	P_{0}
	&=\left(
	\begin{matrix}
	1      &0           &0      & 0               &0      &0             &\cdots    & 0       \\
	0      &1           &0      & 0               &0      &0             &\cdots    & 0       \\
	0      &0           &1      & 0               &0      &0             &\cdots    & 0       \\
	0      &0           &0      &1                &0      &0             &\cdots    & 0       \\
	0      &0           &0      &0                &1      &0             &\cdots     & 0
	\end{matrix}
	\right)_{5\times d'} \text{ and } \ \ T=\left(
	\begin{matrix}
	0      &1        &0      &\cdots       &0              \\
	0      &0        &1      &\cdots       &0              \\
	\vdots &\vdots   &\vdots &\ddots       &\vdots          \\
	0      &0        &0      &\cdots       &1                \\
	1      &0        &0      &\cdots       &0
	\end{matrix}
	\right)_{d\times d'}.
	\end{align*}
	For any  $l\in[d']^{*}$, $a\in [5]^{*}$, define a $5\times d'$ matrix $Q_{l}^{a}$ by
	$$ Q_{l}^{a}(i,j)=\left\{
	\begin{array}{lll}
	0               &           & \text{if} \ \ P_{l}(i,j)=0\\
	\xi_{5}^{a(i-1)}    &           & \text{if} \ \ P_{l}(i,j)=1
	\end{array} \right. , $$
	Let $Z_{0}=\{Q_{l}^{a}, l\in[d'-1],a\in[5]^{*}\}$. Define a $5\times d'$ matrix $H_{y}$, $y\in[3]$, where
	\begin{align*}
	H_{y}
	&=\left(
	\begin{matrix}
	h_{y,1} &0       &0      & 0     &0      &0             &\cdots    & 0       \\
	0       &h_{y,2} &0      & 0     &0      &0             &\cdots    & 0       \\
	0       &0       &h_{y,3}& 0     &0      &0             &\cdots    & 0       \\
	0       &0       &0      &h_{y,4}&0      &0             &\cdots    & 0       \\
	0       &0       &0      &0      &h_{y,5}&0             &\cdots    & 0
	\end{matrix}
	\right)_{5\times d'}.
	\end{align*}
	Then $Z_{0}\cup\{H_{1},H_{2},H_{3}\}$ is a $(5(d'-1)+3)$-number USV1B$5$ in $\mathcal{M}_{5\times d'}$.
\end{example}

\begin{remark}
	Proposition 1 in \cite{Wang17} is a special case of Theorem~\ref{pop} for $d=d'$.  Theorem~\ref{pop} provides new UMEBs, which can be seen from Example~\ref{had1}.
	In fact,  if $d'\geq 10,$ then there is a $5(d'-i)$-number UMEB in $\C^{5}\otimes \C^{d'}$, where $i\in[4]$; if $d'=5+r$, $r\in[4]$, then there is a $5(d'-i)$-number UMEB in $\C^{5}\otimes \C^{d'}$, where $i\in[r]$  \cite{LiMS,Guo16,Chen,ZhanJ}. All these UMEBs have number divisible by $5$. But Example~\ref{had1} constructs a UMEB with number $(5(d'-1)+3)$. The construction in Theorem~\ref{pop} is also different from the constructions in Section~\ref{sec:EBk} for the same reason. See Table \ref{Res} about a summary of numbers in UMEBs.
\end{remark}

\begin{table}[]
	\caption{Results about UMEBs in $\C^{d}\otimes \C^{d'}$.}
	\centering
	\renewcommand\tabcolsep{13.0pt}
	\begin{tabular}{ccc}
		\hline\noalign{\smallskip}
		System                          & No. of UMEB                                   & Reference\\
        \noalign{\smallskip}\hline\noalign{\smallskip}
		$d=d'=2$                        & no UMEB                                       & \cite{Brav}\\
		$d=d'=3$                        & 6                                             & \cite{Brav}\\
		$d=d'=4$                        & 12                                            & \cite{Brav}\\
		$d=d'=5$                        & 23                                            &\cite{Wang17}\\
		$\frac{d'}{2}<d<d'$             & $d^{2}$                                       &\cite{Chen}\\
		$d'\geq 2d$                     & $dm$, $d'-m\in[d-1]$                          &\cite{LiMS}\\
		$d'=d+r, r\in [d-1]$            & $dm$, $d'-m\in[r]$                            &\cite{LiMS} \\
		$d'=qd+r, r\in[d-1]$            & $qd^{2}$                                      & \cite{LiMS,ZhanJ}\\
		$d'\geq 2d$                     & $d(d'-i)$, $i\in[d-1]$                        &\cite{Guo16}\\
		$d'=d+r, r\in[d-1]$             & $d(d'-i)$, $i\in[r]$                          &\cite{,Guo16}\\
		$d=5$, $d'\geq 5$               & $5(d'-1)+3$                                  & This paper\\
        \noalign{\smallskip}\hline
	\end{tabular}\label{Res}
\end{table}

\section{Conclusion and discussion}
We proposed three methods to construct SUEB$k$s (UMEBs) in $\C^{d}\otimes \C^{d'}$:  a construction of SUEB$k$s  from SEB$k$s; a recursive construction of SUEB$pk$s in $\C^{pd}\otimes \C^{qd'}$ from SUEB$k$s in $\C^{d}\otimes \C^{d'}$; and a construction  of UMEBs from unextendible partial Hadamard matrices. We also give three examples: a $28$-number SUEB$4$ in $\C^{6}\otimes \C^{7}$;  a $32$-number UMEB in $\C^{4}\otimes \C^{9}$ constructed from a $4$-number UMEB in $\C^{2}\otimes \C^{3}$; and a $(5(d'-1)+3)$-number UMEB in $\C^{5}\otimes \C^{d'}$, respectively.  Our results cover and improve most results in \cite{Chen,LiMS,Wang17,Guo14,Zhang19,Guo16,ZhanJ,Wang14}. We hope that our results would be useful in studying quantum state tomography and cryptographic protocols with fixed Schmidt number.

Although there exist a lot of constructions of UMEBs and SUEB$k$s, there are still many open questions. Are there UMEBs in $\C^{d}\otimes \C^{d}$ when $d=p \ \text{or}\ 2p $, where $p$ is a prime and $p\equiv 3 \pmod 4$? The minimum open case is when $d=7$,  which was said to have been solved in \cite{Wang17}, but there is a mistake since $k_{4}$ can be zero in the construction. Further, what are the minimum and maximum numbers in SUEB$k$s or UMEBs if they exist? Are there SEB$k$s in $\C^{d}\otimes \C^{d'}$ when $k\nmid dd'$?

\begin{acknowledgements}
F. Shi and X. Zhang are supported by NSFC under Grant No. 11771419,  the Fundamental Research Funds for the Central Universities,
and Anhui Initiative in Quantum Information Technologies under Grant No. AHY150200.
Y. Guo is supported by the Natural
Science Foundation of Shanxi Province under Grant No.
201701D121001, the National Natural Science Foundation
of China under Grant No. 11301312, and the Program for
the Outstanding Innovative Teams of Higher Learning Institutions
of Shanxi.
\end{acknowledgements}

% BibTeX users please use one of
%\bibliographystyle{spbasic}      % basic style, author-year citations
%\bibliographystyle{spmpsci}      % mathematics and physical sciences
%\bibliographystyle{spphys}       % APS-like style for physics
%\bibliography{}   % name your BibTeX data base

% Non-BibTeX users please use

\end{document}